\documentclass[a4paper,10pt]{amsart}

\usepackage[utf8]{inputenc}
\usepackage{amsmath}
\usepackage{amsfonts}
\usepackage{amssymb}
\usepackage{graphicx}
\usepackage{epic}
\newtheorem{definition}{Definition}[section]

\newtheorem{proposition}[definition]{Proposition}

\newcommand{\greydot}{\begin{picture}(20,20)
  \put(0,0){\circle{14}}
  \put(-5,-5){\line(1,1){10}}
  \put(-5,5){\line(1,-1){10}}
  \end{picture}}
\newcommand{\vertexbn}{\begin{picture}(20,20)
  \put(6,-3){\line(1,0){30}}\put(6,3){\line(1,0){30}}
  \put(27,0){\line(-1,1){10}}\put(27,0){\line(-1,-1){10}}
  \end{picture}}
\newcommand{\vertexcn}{\begin{picture}(20,20)
  \put(6,-3){\line(1,0){30}}\put(6,3){\line(1,0){30}}
  \put(17,0){\line(1,-1){10}}\put(17,0){\line(1,1){10}}
  \end{picture}}
\newcommand{\vertexdn}{\begin{picture}(20,40)
  \put(5,5){\line(3,2){19}} \put(5,-5){\line(3,-2){19}}
  \put(31,20){\circle{14}}\put(31,-20){\circle{14}}
  \end{picture}}

\newcommand{\vertexgcn}{\begin{picture}(20,20)
  \put(6,-4){\line(1,0){30}}\put(6,4){\line(1,0){30}}
  \put(17,0){\line(1,1){10}}\put(17,0){\line(1,-1){10}}
  \put(7,0){\line(1,0){28}}
  \end{picture}}
\newcommand{\vertexgbn}{\begin{picture}(20,20)
  \put(6,-4){\line(1,0){30}}\put(6,4){\line(1,0){30}}
  \put(27,0){\line(-1,1){10}}\put(27,0){\line(-1,-1){10}}
  \put(7,0){\line(1,0){28}}
  \end{picture}}

\newcommand{\vertexbbn}{\begin{picture}(20,20)
  \put(6,-2){\line(1,0){29}}\put(6,2){\line(1,0){29}}
  \put(3,-6){\line(1,0){35}}\put(3,6){\line(1,0){35}}
  \put(27,0){\line(-1,1){10}}\put(27,0){\line(-1,-1){10}}
  \end{picture}}

\begin{document}
\title{Vogan diagrams of Affine untwisted Kac-Moody superalgebras}
\author{Biswajit Ransingh}
\date{}
\maketitle
\address{Department of Mathematics\\National Institute of Technology\\Rourkela (India)\\email- bransingh@gmail.com}
\begin{abstract}
 This article classifies the Vogan diagram of the affine untwisted Kac Moody superalgebras.
\end{abstract}

\section{Introduction}

Real forms of Lie superalgebras have a growing application in superstring theory, M-theory and other branches of theoretical physics.
 Magic triangle of M-theory by Satake diagram has been obainted by \cite{borcherd:mtheory}. Similarly supergravity 
 theory can be obatined by Vogan diagrams. 
 Symmetric spaces with the connection of real form of affine Kac-Moody algebras already studied by Vogan diagrams.
Our future work will be  in exploring Symetric superspaces of affine Kac-Moody superalgebras using Vogan diagram.

The last two decades shows a gradual advancement in classification of real form of semisimple Lie algebras to
Lie superalgebras by Satake diagrams and Vogan diagrams. Splits Cartan subalgebra based on Satake diagram  where as Vogan diagram based on
maximally compact Cartan subalgebra. Batra developed the Vogan diagram of affine untwisted kac-Moody algebras \cite{batra:affine,batra:vogan}.
Here we extend the notion to superalgebra case.

If $\mathfrak{g}$ is a complex semisimple Lie algebra with Killing form $B$ and Dynkin diagram $D$, its
real forms $\mathfrak{g}_{\mathbb{R}}\subset \mathfrak{g}$ can be characterized by the Cartan involutions
$$\theta:\mathfrak{g}_{\mathbb{R}}\rightarrow\mathfrak{g}_{\mathbb{R}}$$ The bilinear form  $B(.,\theta)$ 
is symmetric negative definite.
The Vogan diagram denoted by  $(p,d)$, where $d$ is a diagram involution on $D$ and $p$ is a painting on the vertices fixed by $d$.
It is extended  to
Vogan superdiagrams on its extended Dynkin diagram\cite{chuah:vsuper}. Here we extend the theory to affine untwisted Kac Moody 
superalgebras. The future version of the article will also contain the Vogan diagram of twisted Kac-Moody superalgebras.

\section{Cartan Involution and Invariant bilinear form}
An involution $\theta$ of a real semisimple Lie algebra $\mathfrak{g}_0$ such that symmetric bilinear form 
$$B_{\theta}(X,Y)=-B(X,\theta Y) $$ is strictly positive definite is called a Cartan involution.

For Contragradient Lie superalgebras there exist a supersymmetric nondegenerate invariant bilinear form on it and defined in \cite{chuah:vsuper} as
$$B_{\theta}(X,Y)=B(\theta X,\theta Y)$$

Let $\mathcal{G}^{1}$ be a complex affine Kac-Moody superalgebra. The uniqueness of $B$ is extended to $\mathcal{G}^{1}$. 
The killing form is unique by when restricted to $\mathcal{G}_{0}$

An involution $\theta$ for affine Kac-Moody superalgebras is defined by taking identity on $t^{m}$
$$\theta(t^{m}\otimes x)=t^{m}\otimes \theta(x)$$ 
$$\theta (c)=c$$ and $$\theta(d)=d$$
We say a real form has Cartan automorphism $\theta$ if $B$ restrict to the Killing form on $t^{m}\otimes X$ where $X\in\mathcal{G}_{0} $ 
and $B_{\theta}$ is symmetric negative definite on $\mathcal{G}^{1}$.
A bilinear supersymmetric invariant form $B^{1}(,)$ can be set up on $\mathcal{G}^{1}$ by the definitions
$$B^{(1)}(t^{j}\otimes X,t^{k}\otimes Y)=\delta^{j+k}B(X,Y)$$
$$B^{(1)}(t^{j}\otimes X,C)=0$$
$$B^{(1)}(t^{j}\otimes X,D)=0$$
$$B^{(1)}(C,C)=0$$
$$B^{(1)}(C,D)=1$$
$$B^{(1)}(D,D)=0$$

\begin{proposition}

 Let $\theta\in$$aut$$_{2,4}(\mathcal{G}^{1})$. There exists a real form $\mathcal{G}^{1}_{\mathbb{R}}$ such that $\theta$ restricts
 to a Cartan automorphism on $\mathcal{G}^{1}_{\mathbb{R}}$.
\end{proposition}

\begin{proof}

 Since $\theta$ is an $\mathcal{G}^{1}$ automorphism, it preserves $B$. namely
 $$B(X,Y)=B(\theta X,\theta Y)$$
 $B_{\theta}(X,Y)=B_{\theta}(Y,X)$, $B_{\theta}(X,Y)=B_{\theta}(\theta X,\theta Y)$, $B_{\theta}(X.\theta X)=0$
$$B_{\theta}(X\otimes t^{m},Y\otimes t^{n})=B_{\theta}(Y\otimes t^{n},X\otimes t^{m})=$$
$$B(X\otimes t^{m},Y\otimes t^{n})=t^{m+n}B(X,Y)$$
 for all $X,Y\in\mathcal{G}_{0}$
 $$B(K,X\otimes t^{k})=B(D,X\otimes t^{k})=B(D,D)=B(K,K)=0$$
 For $z\in L(t,t^{-1})\otimes\mathcal{G}_0$ and $X,Y\in L(t,t^{-1})\otimes\mathcal{G}_1$ 
 $$B_{\theta}(X,[Z,Y])=B(X,[\theta Z,\theta Y])=-B_{\theta}(X,[\theta Z,\theta Y])$$
 $$B_{\theta}(X,[Z,Y])=0$$ $\forall$ $X\in\mathbb{C}c$ or $\mathbb{C}d$  
 
 $\mathcal{G}_{\mathbb{R}}^{(1)}\simeq\mathcal{G}_{\overline{0}\mathbb{R}}^{(1)}\simeq\mathcal{G}_{\overline{0}\mathbb{R}}$.
The above three real forms are isomorphic. So the Cartan decomposition of $\mathcal{G}_{\mathbb{R}}^{(1)}$
are isomorphic

to $\mathcal{G}_{\overline{0}}$. 

$\mathcal{G}_{\overline{0}}=\mathfrak{k}_{0}\oplus\mathfrak{p}_{0}$

$B_{\theta}(X,[Z,Y])=\begin{cases}
\begin{array}{cc}
-B_{\theta}([Z,X],Y) & \mbox{if }Z\in\mathfrak{k}_{0}\\
B_{\theta}([Z,X],Y) & \mbox{if }Z\in\mathfrak{p}_{0}
\end{array}\end{cases}$

We say that a real form of $\mathcal{G}$ has Cartan automorphism $\theta\in$aut$_{2,4}(\mathcal{G})$
if $B$ restricts to the Killing form on $\mathcal{G}_0$ and $B_\theta$ is symmetric negative definite on 
$\mathcal{G}_{\mathbb{R}}$.
 $B_{\theta}(X_{i},X_{j})=\delta_{ij}$. It follows that $B_{\theta}$ negative definite on $L(t,t^{-1})\otimes \mathcal{G}^{(1)}$.
 By $B_{\theta}$ is symmetric bilinear form on $L_{1}$ $\{1\otimes X_{1},1\otimes X_{2},\cdots,d\}$. So it is conclude that $\theta$ is a 
 Cartan automorphism on $\mathcal{G}^{(1)}$.
\end{proof}
 \section{Vogan diagram}
A root is real if it takes on real values on $\mathfrak{h}_{0}$ ,(i.e., vanishes on $\mathfrak{a}_{0}$) 
imaginary if it takes on purely imaginary values on $\mathfrak{h}_{0}$ (i.e., vanishes on $\mathfrak{a}_{0}$)
and complex otherwise.
A $\theta$ stable Cartan subalgebra $\mathfrak{h}_{0}=\mathfrak{k}_{0}\oplus\mathfrak{a}_{0}$
is maximally compact if its compact dimension is as large as possible, maximally noncompact if its 
noncompact dimensional is as large as possible.
An imaginary root $\alpha$ compact if $\mathfrak{g}_{\alpha}\subseteq \mathfrak{k}$, noncompact if
$\mathfrak{g}_{\alpha}\subseteq \mathfrak{p}$
Let $\mathfrak{g}_0$ be a real semisimple Lie algebra, Let  $\mathfrak{g}$
be its complexification, let $\theta$ be a Cartan involution, let 
$\mathfrak{g}_{0}=\mathfrak{k}_{0}\oplus\mathfrak{p}_{0}$ be the corresponding Cartan decomposition
A maximally compact $\theta$ stable Cartan subalgebra $\mathfrak{h}_{0}=\mathfrak{k}_{0}\oplus\mathfrak{p}_{0}$
of $\mathfrak{g}_0$ with complexification $\mathfrak{h}=\mathfrak{k}\oplus\mathfrak{p}$ and we let 
$\triangle=\triangle(\mathfrak{g},\mathfrak{h})$ be the set of roots.
Choose a positive system $\triangle^{+}$  for $\triangle$ that takes $i\mathfrak{t}_0$ before $\mathfrak{a}$.
$\theta(\triangle^{+})=\triangle^{+}$\\
$\theta(\mathfrak{h}_{0})=\mathfrak{k}_{0}\oplus(-1)\mathfrak{p}_{0}$. 
Therefore $\theta$ permutes the simple roots. It must fix the simple roots that are imaginary and permute
in 2-cycles the simple roots that are complex.
By the Vogan diagram of the triple $(\mathfrak{g}_{0},\mathfrak{h}_{0},\triangle^{+}))$., we mean the Dynkin diagram of 
$\triangle^{+}$ with the 2 element orbits under $\theta$ so labeled and with the 1-element orbits painted or not, according as
the corresponding imaginary simple root is noncompact or compact.

The uniqueness of Cartan automorphism from Dynkin diagram of $\mathcal{G}_{0}$ to $\mathcal{G}^{(1)}$ proved in \cite{chuah:vsuper}. 
This gives  a straightforward proof of the above theory to affine untwisted Kac-Moody superalgebras cases with the addtion of 
canonical central element $K$ and derivation $D$.
\begin{definition}
 A  Vogan diagram $(p,d)$ on $D$ of $\mathcal{G}^{(1)}$ and one of the following holds:
 \begin{itemize}
  \item [(i)] $\theta$ fixes grey vertices
  \item[(ii)] $\theta$ interchange grey vertices and $\sum_{S}a_\alpha$ is odd.
  \item[(iii)] $\sum_{S}a_\alpha$ is odd
 \end{itemize}
\end{definition}
\begin{proposition}
 Let $\mathcal{G}_{\mathbb{R}}$ be a real form, with Cartan involution $\theta\in $inv$(\mathcal{G}_{\mathbb{R}})$ and Vogan diagram
 $(p,d)$ of $D_{0}$. The following are equivalent
 \begin{itemize}
  \item[(i)] $\theta$ extend to aut$_{2,4}\mathcal{G}^{(1)}$.
  \item[(ii)] $(\mathcal{G}_{\bar{0}\mathbb{R}})$ extend to a real form of $\mathcal{G}^{(1)}$.
  \item [(iii)] $(p,d)$ extend to a Vogan diagram on $D$
 \end{itemize}
\end{proposition}
\begin{proof}
Let $S$ be the $d-$ orbits of vertices defined by \cite{Chuah:finite} 
$$S=$$\{vertices painted by p\}$$\cup$$\{white and adjacent 2-element d-orbits\}$$\cup$$\{grey and non adjacent 2-element d-orbits\}$$$$
 Let $D$ be the Dynkin diagram of $\mathcal{G}^{(1)})$ of simple root system
 $\Phi\cup\phi$($\Phi$ simple root system with $\phi$ lowest root) with $D=D_{\bar{0}}+D_{\bar{1}}$, where $D_{\bar {0}}$
 and $D_{\bar{1}}$ are respectively the white and grey vertices. The numerical label of the diagram shows $\sum_{\alpha\in D_{\bar{1}}}=2$ has either
 two grey vertices with label 1 or one grey vertex with label 2.
 
 \begin{itemize}
  \item[(i)] $D_{\bar{1}}=\{\gamma,\delta\}$ so the labelling of the odd vertices are 1.
  \item[(ii)] $D_{\bar{1}}=\{\gamma\}$ so labelling is 2 $( a_{\alpha}=2)$ on odd vertex.
 \end{itemize}
  $\theta\in $inv$(\mathcal{G}_{\mathbb{R}})$;  
 $\theta$ permutes the weightspaces $L(t,t^{-1})\otimes\mathcal{G}_{\bar{1}}$
 The rest part of proof of the proposition is followed the proof of the propostion 2.2 of \cite{chuah:vsuper}
 \end{proof}

\section{Affine Kac-Moody superalgebras}

   Let a finite and countable set $I = \{ 1,\dots,r \}$ with  $\tau\subset I$. To a given generalized
  Cartan matrix $A$ and subset $\tau$, there exist a  Lie superalgebra $\bar{\mathcal{G}}(A,\tau)$  
   with the following set of  relations
  \begin{eqnarray*}
    && [ h_i,h_j ] = 0  \\
    && [e_{i},f_{j}]=\delta_{ij}h_{i}\\
    && [ h_i,e_j ] =  a_{ij} e_j \\
    && [ h_i,f_j ] =  -a_{ij} f_j \\
    && deg (h_{i})=deg (f_{i})=\bar{1}  if\, i\in\tau\\
    && deg (h_{i})=deg (f_{i})=\bar{0}  if\, i\notin\tau 
   \end{eqnarray*}
   Let $e_{ij}=(ad e_{i})^{1-\frac{2a_{ij}}{a_{ii}}}e_j$ and $f_{ij}=(ad f_{i})^{1-\frac{2a_{ij}}{a_{ii}}}f_j$
   
  We have the triangular decomposition of 
  $$\bar{\mathcal{G}}(A,\tau)=N^{-}_{f_{i}}\oplus H_{h_{i}}\oplus N^{+}_{f_{i}}$$ 
  Let the ideal of $N^{-}_{f_{i}}$ generated by $[f_{i},f_{j}]$ is $R^{-}$ and the ideal
  generated $[e_{i},e_{j}]$ by $N^{+}_{f_{i}}$ is $R^{+}$ such that $a_{ij}=0$ and all the 
 $f_{ij}$ and  $e_{ij}$for the former and later respectivily.
  $R=R^{+}\otimes R^{-}$ is an ideal of $\mathcal{G}(A,\tau)$\cite{ray:char}. The quotient  
  $\bar{\mathcal{G}}(A,\tau)/R=\mathcal{G}(A,\tau)$  is called a generalised Kac-Moody superalgebra.
\begin{itemize}
 \item[(a)] $\mathcal{G}^{(1)}$ is an affine Kac-Moody superalgebra if $A$ is indecomposable.
 \item [(b)]There exists a vector $(a_{i})_{i=1}^{m+n}$, with $a_{i}$ all positive such that $A(a_{i})_{i=1}^{m+n}$=0. Then $A$
is called Cartan matrix of affine type. The Affine superalgebra associted with a generalized Cartan matrix of type $X^{1}(m,n)$
is called untwised affine Kac-Moody
superalgebra.
\end{itemize}

\medskip
\subsection{Dynkin diagram associated with a generalised Cartan matrix(GCM)}
The Kac-Moody superalgebra $\mathcal{G}(A,\tau)$ is associated with a Dynkin diagram
according to the following rules. Taking the assumption that $i \in \tau$ if $a_{ii} = 0$.

  From a  $GCM$ $A$ with  each $i$ of the diagonal entries ($a_{ii}$) 2 and $i \notin \tau$ a
    white dot  and $i \in \tau$ a black dot $\bullet$, to each $i$ such that $a_{ii}=0$ and $i \in \tau$ a grey dot $\otimes$.
        The $i$-th and $j$-th roots will be joined by $\zeta_{ij}=\max\Big( |a_{ij}|,|a_{ji}| \Big)$ 
        lines with $|a_{ij} \, a_{ji}| \le4$ and 
        the off digonal entries nonzero
        where for off diagonal entries zero; then the  number of connection lines are $|a_{ij}|=|a_{ji}|$ with
        $|a_{ij}|$ and $|a_{ji}| \le 4$
      
    The  arrows will be added  on the lines connecting the $i$-th and $j$-th dots 
    when $\zeta_{ij} > 1$ and $|a_{ij}| \ne |a_{ji}|$, pointing from $j$ 
    to $i$ if $|a_{ij}| > 1$.
 One can get the different Dynkin diagrams with details in \cite{Frap:affine,frap:hypersu,kac:sup}.

\section{A Realization of Affine Kac-Moody superalgebras}
Let $L=\mathbb{C}[t,t^{1}]$ be an algebra of Laurent polynomial in t. The residue of a Laurent polynomial
$P=\sum_{k\in \mathbb{Z}}c_{k}^k$ (where all but a finite number of $c_k$ are 0) is defined as $Res P=c_{-1}$.
Let $\mathcal{G}$ be a simple Lie superalgebra.
Let $\mathcal{G}$ be a finite dimensional simple Lie superalgebra ($\mathcal{G}\neq gl(n|n)$,), $(.|.)$ be a  nondegenerate 
invariant symmetric bilinear form on $\mathcal{G}$.
The definition of affine untwisted B.S.A. $\mathcal{G}^{(1)}$ follows that of affine algebras, i.e.
$\mathcal{G}^{(1)}$ is the loop algebra constructed from $\mathcal{G}$.
Define an infinite dimensional superalgebra $\mathcal{G}^{(1)}$
as $\mathcal{G}\otimes\mathbb{C}[t,t^{-1}]\oplus\mathbb{C}D\oplus\mathbb{C}K$
here $D,K$ are even elements and bracket is defined by
$$[X\otimes t^{k},Y\otimes t^{l}]=[X,Y]\otimes t^{k+l}+k\delta_{k,-l}(X,Y)K,$$
$$[D,K]=0$$, $$[D,X\otimes t^{k}]=kX\otimes t^{k}$$

Untwisted Affine B.S.A. Properties on the structure of affine Lie superalgebras can also be deduced by
extending the classification of Dynkin diagrams to the affine case. This will in particular allow us
to construct in a diagrammatic way twisted affine superalgebras from untwisted ones. 

A simple root system of an affine B.S.A. $\mathcal{G}^{(1)}$ is obtained from a simple root system B
of $\mathcal{G}$ by adding to it the affine root which project on $B$ as the corresponding lowest root.
The simple root systems of $\mathcal{G}^{(1)}$ are therefore associated to the extended Dynkin diagrams
used to determine the regular subsuperalgebras.

\section{Root of $\mathcal{G}^{1}$}

\begin{itemize}

 \item[(i)] $A^{(1)}(m,n)=spl^{(1)}(m+1,n+1)$

 $\Phi\cup\phi=\{\alpha_{0}=k+\delta_{n+1}-e_{1},\alpha_{1}=e_{1}-e_{2},\cdots,\alpha_{m}=e_{m}-e_{m+1},\alpha_{m+1}=
 e_{m+1}-\delta_{1},\alpha_{m+2}=\delta_{1}-\delta_{2},\cdots,\alpha_{n+m+1}=\delta_{n}-\delta_{n+1}\}$
 
 \item[(ii)] $B^{(1)}(m,n)=osp^{(1)}(2m+1,2n)(m>2)$
 
 $\Phi\cup\phi=\{k-2\delta_{1},\alpha_{1}=\delta_{1}-\delta_{2},\alpha_{2}=\delta_{2}-\delta_{3}
 ,\cdots,\alpha_{n}=\delta_{n}-e_{1},\alpha_{n+1}=e_{1}-e_{2},\alpha_{n+m+1}=e_{m-1}-e_{m},\alpha_{n+m}=e_{m}$
 
\item[(iii)] $D^{(1)}(m+n)=osp^{(1)}(2m,2n)(m>2)$

$\Phi\cup\phi=\{k-2\delta_{1},\alpha_{1}=\delta_{1}-\delta_{2},\alpha_{2}=\delta_{2}-\delta_{3}
 ,\cdots,\alpha_{n}=\delta_{n}-e_{1},\alpha_{n+1}=e_{1}-e_{2},\alpha_{n+m-1}=e_{m-1}-e_{m},\alpha_{n+m-1}=e_{m-1}+e_{m}$
 
\item[(iv)] $C^{(1)}(n)$
 
 $\Phi\cup\phi=\{\alpha_{0}=k-e-\delta_{1},\alpha_{1}=e-\delta_{1},\alpha_{2}=\delta_{1}-\delta_{2},\cdots
 \alpha_{n}=\delta_{n-1}-\delta_{n},\alpha_{n+1}=2\delta_{n-1}\}$
 
 \item[(v)] $D^{(1)}(2,1,\alpha)$
 
 $\Phi\cup\phi=\{\alpha_{0}=k-(e_{1}+e_{2}+e_{3}),e_{1}-e_{2}-e_{3},2e_{2},2e_{3}\}$
 
 \item[(vi)] $F^{(1)}(4)$
 
 $\Phi\cup\phi=\{\alpha_{0}=k-3\delta,\delta+\frac{1}{2}(-e_{1}-e_{2}-e_{3}),e_{3},e_{2}-e_{3},e_{1}-e_{2})\}$
 
 \item[(vii)] $G^{(1)}(3)$
 
 $\Phi\cup\phi=\{\alpha_{0}=k-4\delta,\delta+e_{1},e_{2},e_{3}-e_{2}\}$
\end{itemize}
The Cartan subalgebra of  $\mathcal{G}^{1}$ is $$\mathfrak{h}=\mathring{\mathfrak{h}\oplus\mathbb{C}K\oplus
\mathbb{C}D}$$

\section{Real forms from Vogan diagram of affine untwisted Kac-Moody superalgebras}

 $A^{(1)}(m,n)$
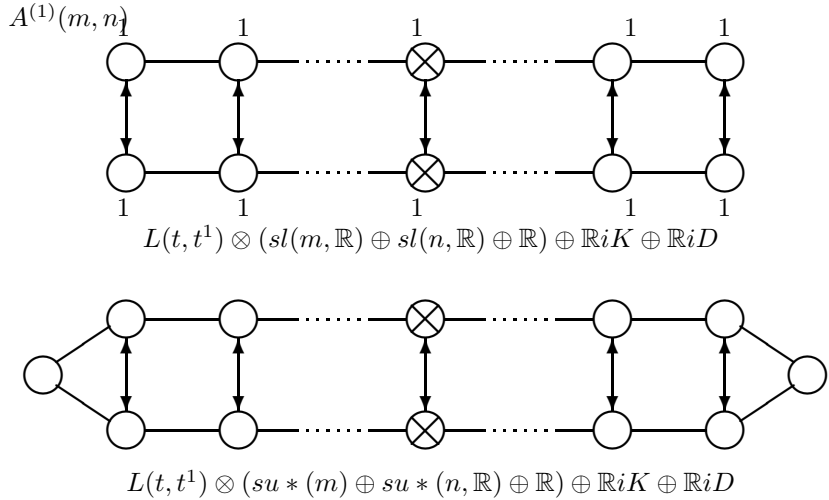
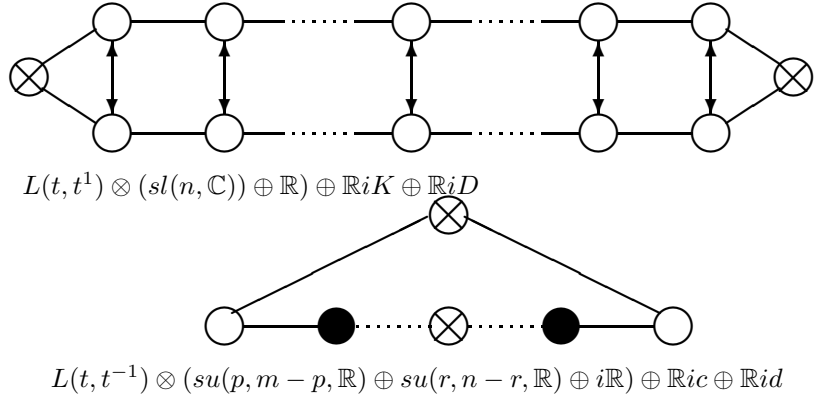
 \begin{displaymath}
 \begin{picture}(80,50) \thicklines
\put(-84,0){\circle{14}} \put(-42,0){\circle{14}}\put(28,0){\greydot}\put(98,0){\circle{14}}\put(140,0){\circle{14}}

 \put(-77,0){\line(1,0){28}} \put(-35,0){\line(1,0){14}}\put(7,0){\line(1,0){14}}\put(35,0){\line(1,0){14}}\put(-21,0){\dottedline{4}(1,0)(28,0)}
 \put(49,0){\dottedline{4}(1,0)(28,0)}\put(77,0)
 {\line(1,0){14}}
 \put(105,0){\line(1,0){28}}
  \put(-84,42){\circle{14}} \put(-42,42){\circle{14}}\put(28,42){\greydot}\put(98,42){\circle{14}}\put(140,42){\circle{14}}
  
 \put(-77,42){\line(1,0){28}}  \put(-35,42){\line(1,0){14}} \put(-21,42){\dottedline{4}(1,0)(28,0)}\put(7,42){\line(1,0){14}}\put(35,42){\line(1,0){14}}
 \put(49,42){\dottedline{4}(1,0)(28,0)}
 \put(77,42){\line(1,0){14}} \put(105,42){\line(1,0){28}}
 
 \put(-84,7){\line(0,1){28}}
 \put(-42,7){\line(0,1){28}}\put(28,7){\line(0,1){28}}\put(98,7){\line(0,1){28}}\put(140,7){\line(0,1){28}}
 
 \put(30,-25){\makebox(0,0){$L(t,t^{1})\otimes (sl(m,\mathbb{R})\oplus sl(n,\mathbb{R})\oplus \mathbb{R})\oplus \mathbb{R}iK\oplus 
 \mathbb{R}iD$}}
  \put(-40,55){\makebox(0,0){$1$}}\put(25,55){\makebox(0,0){$1$}}\put(105,55){\makebox(0,0){$1$}}\put(140,55){\makebox(0,0){$1$}}
 \put(-85,55){\makebox(0,0){$1$}}
 \put(-40,-13){\makebox(0,0){$1$}}\put(25,-13){\makebox(0,0){$1$}}\put(105,-13){\makebox(0,0){$1$}}\put(140,-13){\makebox(0,0){$1$}}
 \put(-85,-13){\makebox(0,0){$1$}}
 \put(-84,35){\vector( 0, 1){0}}\put(-84,7){\vector( 0,-1){0}}\put(28,35){\vector( 0, 1){0}}\put(28,7){\vector( 0,-1){0}}
 \put(-42,35){\vector( 0, 1){0}}\put(-42,7){\vector( 0,-1){0}} \put(98,35){\vector( 0, 1){0}}\put(98,7){\vector( 0,-1){0}}
  \put(140,35){\vector( 0, 1){0}}\put(140,7){\vector( 0,-1){0}}
\end{picture}
   \end{displaymath}
   
   \vspace{1cm}
   \begin{displaymath}
      \begin{picture}(80,50) \thicklines
 \put(-115,21){\circle{14}}\put(-110,26){\line(3,2){20}}\put(-110,17){\line(3,-2){20}}
 \put(171,21){\circle{14}}\put(166,26){\line(-3,2){20}}\put(166,17){\line(-3,-2){20}}
\put(-84,0){\circle{14}} \put(-42,0){\circle{14}}\put(28,0){\greydot}\put(98,0){\circle{14}}\put(140,0){\circle{14}}

 \put(-77,0){\line(1,0){28}} \put(-35,0){\line(1,0){14}}\put(7,0){\line(1,0){14}}\put(35,0){\line(1,0){14}}\put(-21,0){\dottedline{4}(1,0)(28,0)}
 \put(49,0){\dottedline{4}(1,0)(28,0)}\put(77,0)
 {\line(1,0){14}}
 \put(105,0){\line(1,0){28}}
  \put(-84,42){\circle{14}} \put(-42,42){\circle{14}}\put(28,42){\greydot}\put(98,42){\circle{14}}\put(140,42){\circle{14}}
  
 \put(-77,42){\line(1,0){28}}  \put(-35,42){\line(1,0){14}} \put(-21,42){\dottedline{4}(1,0)(28,0)}\put(7,42){\line(1,0){14}}\put(35,42){\line(1,0){14}}
 \put(49,42){\dottedline{4}(1,0)(28,0)}
 \put(77,42){\line(1,0){14}} \put(105,42){\line(1,0){28}}
 
 \put(-84,7){\line(0,1){28}}
 \put(-42,7){\line(0,1){28}}\put(28,7){\line(0,1){28}}\put(98,7){\line(0,1){28}}\put(140,7){\line(0,1){28}}
  \put(30,-20){\makebox(0,0){$L(t,t^{1})\otimes (su*(m)\oplus su*(n,\mathbb{R})\oplus \mathbb{R})\oplus \mathbb{R}iK\oplus 
 \mathbb{R}iD$}}
  \put(-84,35){\vector( 0, 1){0}}\put(-84,7){\vector( 0,-1){0}}\put(28,35){\vector( 0, 1){0}}\put(28,7){\vector( 0,-1){0}}
 \put(-42,35){\vector( 0, 1){0}}\put(-42,7){\vector( 0,-1){0}} \put(98,35){\vector( 0, 1){0}}\put(98,7){\vector( 0,-1){0}}
  \put(140,35){\vector( 0, 1){0}}\put(140,7){\vector( 0,-1){0}}
        \end{picture}
   \end{displaymath}
   \vspace{1cm}
   \begin{displaymath}
      \begin{picture}(80,50) \thicklines
 \put(-115,21){\greydot}\put(-110,26){\line(3,2){20}}\put(-110,17){\line(3,-2){20}}
 \put(171,21){\greydot}\put(166,26){\line(-3,2){20}}\put(166,17){\line(-3,-2){20}}
\put(-84,0){\circle{14}} \put(-42,0){\circle{14}}\put(28,0){\circle{14}}\put(98,0){\circle{14}}\put(140,0){\circle{14}}

 \put(-77,0){\line(1,0){28}} \put(-35,0){\line(1,0){14}}\put(7,0){\line(1,0){14}}\put(35,0){\line(1,0){14}}\put(-21,0){\dottedline{4}(1,0)(28,0)}
 \put(49,0){\dottedline{4}(1,0)(28,0)}\put(77,0)
 {\line(1,0){14}}
 \put(105,0){\line(1,0){28}}
  \put(-84,42){\circle{14}} \put(-42,42){\circle{14}}\put(28,42){\circle{14}}\put(98,42){\circle{14}}\put(140,42){\circle{14}}
  
 \put(-77,42){\line(1,0){28}}  \put(-35,42){\line(1,0){14}} \put(-21,42){\dottedline{4}(1,0)(28,0)}\put(7,42){\line(1,0){14}}\put(35,42){\line(1,0){14}}
 \put(49,42){\dottedline{4}(1,0)(28,0)}
 \put(77,42){\line(1,0){14}} \put(105,42){\line(1,0){28}}
 
 \put(-84,7){\line(0,1){28}}
 \put(-42,7){\line(0,1){28}}\put(28,7){\line(0,1){28}}\put(98,7){\line(0,1){28}}\put(140,7){\line(0,1){28}}

 \put(-32,-20){\makebox(0,0){$L(t,t^{1})\otimes (sl(n,\mathbb{C}))\oplus \mathbb{R})\oplus \mathbb{R}iK\oplus 
 \mathbb{R}iD$}}
  \put(-84,35){\vector( 0, 1){0}}\put(-84,7){\vector( 0,-1){0}}\put(28,35){\vector( 0, 1){0}}\put(28,7){\vector( 0,-1){0}}
 \put(-42,35){\vector( 0, 1){0}}\put(-42,7){\vector( 0,-1){0}} \put(98,35){\vector( 0, 1){0}}\put(98,7){\vector( 0,-1){0}}
  \put(140,35){\vector( 0, 1){0}}\put(140,7){\vector( 0,-1){0}}
        \end{picture}
   \end{displaymath}
 \begin{center}

    \begin{displaymath}
    \begin{array}{ccc}
\begin{picture}(80,50) \thicklines 
   \put(-42,0){\circle{14}}   \put(0,0){\circle*{14}}\put(42,42){\greydot}
   \put(42,0){\greydot} \put(84,0){\circle*{14}} \put(126,0){\circle{14}} 
   
   \put(-39,5){\line(2,1){73}}  
    \put(-35,0){\line(1,0){28}}  \put(7,0){\dottedline{4}(1,0)(28,0)} \put(49,0){\dottedline{4}(1,0)(28,0)} \put(91,0){\line(1,0){28}} 
    
    \put(121,5){\line(-2,1){73}} 
      \put(42,20){\makebox(0,0){}}
      \put(30,-20){\makebox(0,0){$L(t,t^{-1})\otimes (su(p,m-p,\mathbb{R})\oplus su(r,n-r,\mathbb{R})\oplus i\mathbb{R})
      \oplus \mathbb{R}ic\oplus\mathbb{R}id$}}
      \put(84,-20){\makebox(0,0){}}
         \end{picture}
\end{array}
\end{displaymath}
   \end{center}
   \vspace{2cm}
   
Case $B^{(1)}(m,n)=Osp^{(1)}(2m+1,2n)$

The  below first Vogan diagram which contains the extreme right black painted root is from the original Dynkin diagram color.
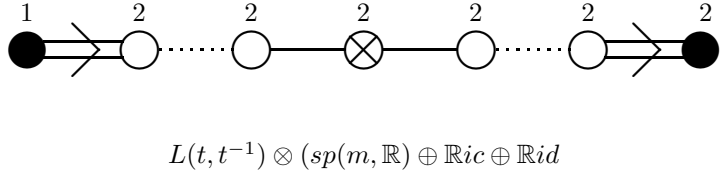
\begin{displaymath}
 \begin{picture}(80,50) \thicklines 
 \put(-84,0){\circle*{14}}   \put(-42,0){\circle{14}}   \put(0,0){\circle{14}} \put(42,0){\greydot} \put(84,0){\circle{14}}
 \put(126,0){\circle{14}} \put(168,0){\circle*{14}} 
    \put(-84,0){\vertexbn}\put(-35,0){\dottedline{4}(1,0)(28,0)}\put(7,0){\line(1,0){28}} \put(49,0){\line(1,0){28}}
    \put(91,0){\dottedline{4}(1,0)(28,0)}
    \put(126,0){\vertexbn}
      \put(0,-20){\makebox(0,0){}} 
      \put(42,-20){\makebox(0,0){}} 
      \put(84,-20){\makebox(0,0){}} 
        \put(-84,14){\makebox(0,0){$1$}}
      \put(-42,14){\makebox(0,0){$2$}}
      \put(0,14){\makebox(0,0){$2$}} 
      \put(42,14){\makebox(0,0){$2$}} 
      \put(84,14){\makebox(0,0){$2$}}
       \put(126,14){\makebox(0,0){$2$}}
       \put(170,14){\makebox(0,0){$2$}}
       \put(42,-40){\makebox(0,0){$L(t,t^{-1})\otimes (sp(m,\mathbb{R})
      \oplus \mathbb{R}ic\oplus\mathbb{R}id$}}
    \end{picture}
\end{displaymath}
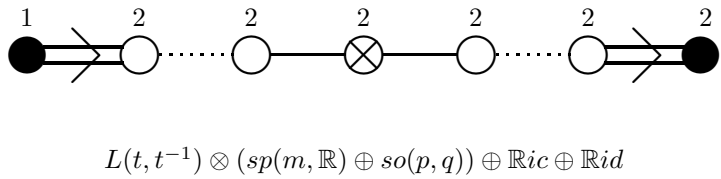
\vspace{2cm}

\begin{displaymath}
 \begin{picture}(80,50) \thicklines 
 \put(-84,0){\circle*{14}}   \put(-42,0){\circle{14}}   \put(0,0){\circle{14}} \put(42,0){\greydot} \put(84,0){\circle{14}}
 \put(126,0){\circle{14}} \put(168,0){\circle*{14}} 
    \put(-84,0){\vertexbn}\put(-35,0){\dottedline{4}(1,0)(28,0)}\put(7,0){\line(1,0){28}} \put(49,0){\line(1,0){28}}
    \put(91,0){\dottedline{4}(1,0)(28,0)}
    \put(126,0){\vertexbn}
      \put(0,-20){\makebox(0,0){}} 
      \put(42,-20){\makebox(0,0){}} 
      \put(84,-20){\makebox(0,0){}} 
        \put(-84,14){\makebox(0,0){$1$}}
      \put(-42,14){\makebox(0,0){$2$}}
      \put(0,14){\makebox(0,0){$2$}} 
      \put(42,14){\makebox(0,0){$2$}} 
      \put(84,14){\makebox(0,0){$2$}}
       \put(126,14){\makebox(0,0){$2$}}
       \put(170,14){\makebox(0,0){$2$}}
       \put(42,-40){\makebox(0,0){$L(t,t^{-1})\otimes (sp(m,\mathbb{R})\oplus so(p,q))
      \oplus \mathbb{R}ic\oplus\mathbb{R}id$}}
    \end{picture}
\end{displaymath}
\vspace{2cm}

Case $B^{(1)}(0,n)=Osp^{(1)}(1,2n)$
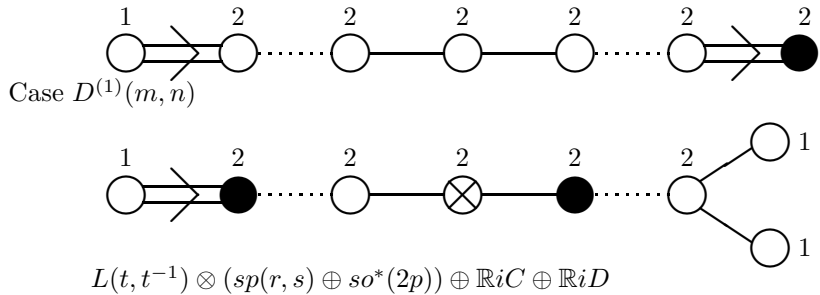
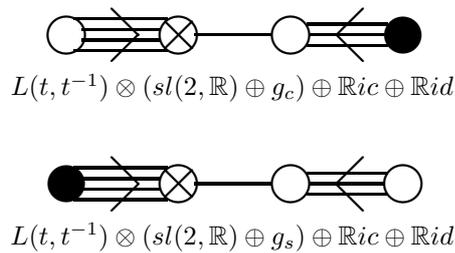
\begin{center}
\begin{displaymath}
 \begin{picture}(80,50) \thicklines 
 \put(-84,0){\circle{14}}   \put(-42,0){\circle{14}} \put(0,0){\circle{14}} \put(42,0){\circle{14}} 
 \put(84,0){\circle{14}} \put(126,0){\circle{14}}
 \put(168,0){\circle*{14}} 
 
    \put(-84,0){\vertexbn}\put(-35,0){\dottedline{4}(1,0)(28,0)}\put(7,0){\line(1,0){28}} \put(49,0){\line(1,0){28}}
    \put(91,0){\dottedline{4}(1,0)(28,0)}
    \put(126,0){\vertexbn}
    
      \put(0,-20){\makebox(0,0){}} 
      \put(42,-20){\makebox(0,0){}} 
      \put(84,-20){\makebox(0,0){}} 
          \put(-84,14){\makebox(0,0){$1$}}
      \put(-42,14){\makebox(0,0){$2$}}
      \put(0,14){\makebox(0,0){$2$}} 
      \put(42,14){\makebox(0,0){$2$}} 
      \put(84,14){\makebox(0,0){$2$}}
       \put(126,14){\makebox(0,0){$2$}}
       \put(170,14){\makebox(0,0){$2$}}
    \end{picture}
\end{displaymath}

 \end{center}
 
 Case $D^{(1)}(m,n)$
\begin{displaymath}
 \begin{picture}(80,30) \thicklines 
 \put(-84,0){\circle{14}} \put(-42,0){\circle*{14}}   \put(0,0){\circle{14}} \put(42,0){\greydot} \put(84,0){\circle*{14}}
 \put(126,0){\circle{14}}
 
   \put(-84,0){\vertexbn}\put(-35,0){\dottedline{4}(1,0)(28,0)}\put(7,0){\line(1,0){28}} 
   \put(49,0){\line(1,0){28}}\put(91,0){\dottedline{4}(1,0)(28,0)}
   \put(126,0){\vertexdn}
  
      \put(0,-20){\makebox(0,0){}}
      \put(42,-20){\makebox(0,0){}} 
      \put(84,-20){\makebox(0,0){}} 
         \put(-84,14){\makebox(0,0){$1$}}
      \put(-42,14){\makebox(0,0){$2$}}
      \put(0,14){\makebox(0,0){$2$}} 
      \put(42,14){\makebox(0,0){$2$}} 
      \put(84,14){\makebox(0,0){$2$}}
       \put(126,14){\makebox(0,0){$2$}}
       \put(170,-20){\makebox(0,0){$1$}}
       \put(170,20){\makebox(0,0){$1$}}
   \put(0,-32){\makebox(0,0){$L(t,t^{-1})\otimes (sp(r,s)\oplus so^{*}(2p))\oplus \mathbb{R}iC\oplus\mathbb{R}iD$}}     
    \end{picture}
\end{displaymath}
\\

\begin{displaymath}
 \begin{picture}(80,30) \thicklines 
 \put(-84,0){\circle{14}}   \put(-42,0){\circle{14}}   \put(0,0){\circle{14}} \put(42,0){\greydot} \put(84,0){\circle{14}} \put(126,0){\circle{14}}

    \put(-84,0){\vertexbn}\put(-35,0){\dottedline{4}(1,0)(28,0)}\put(7,0){\line(1,0){28}} \put(49,0){\line(1,0){28}}\put(91,0){\dottedline{4}(1,0)(28,0)}
   \put(126,0){\vertexdn}
    \qbezier(176.5, 0)(175, -15)(165, -23)
    \qbezier(176.5, 0)(175, 15)(165, 23)
\put(165,23){\vector( -1, 1){0}}
\put(165,-23){\vector( -1, -1){0}}
      \put(0,-20){\makebox(0,0){}} 
      \put(42,-20){\makebox(0,0){}} 
      \put(84,-20){\makebox(0,0){}} 
         \put(-84,14){\makebox(0,0){$1$}}
      \put(-42,14){\makebox(0,0){$2$}}
      \put(0,14){\makebox(0,0){$2$}} 
      \put(42,14){\makebox(0,0){$2$}} 
      \put(84,14){\makebox(0,0){$2$}}
       \put(126,14){\makebox(0,0){$2$}}
       \put(170,-20){\makebox(0,0){$1$}}
       \put(170,20){\makebox(0,0){$1$}}
    \end{picture}
\end{displaymath}
\\

\begin{displaymath}
 \begin{picture}(80,30) \thicklines 
 \put(-84,0){\circle*{14}}   \put(-42,0){\circle{14}}   \put(0,0){\circle{14}} \put(42,0){\greydot} \put(84,0){\circle{14}}
 \put(126,0){\circle*{14}}

    \put(-84,0){\vertexbn}\put(-35,0){\dottedline{4}(1,0)(28,0)}\put(7,0){\line(1,0){28}} \put(49,0){\line(1,0){28}}
    \put(91,0){\dottedline{4}(1,0)(28,0)}
   \put(126,0){\vertexdn}
    \qbezier(176.5, 0)(175, -15)(165, -23)
    \qbezier(176.5, 0)(175, 15)(165, 23)
\put(165,23){\vector( -1, 1){0}}
\put(165,-23){\vector( -1, -1){0}}
      \put(0,-20){\makebox(0,0){}} 
      \put(42,-20){\makebox(0,0){}} 
      \put(84,-20){\makebox(0,0){}} 
         \put(-84,14){\makebox(0,0){$1$}}
      \put(-42,14){\makebox(0,0){$2$}}
      \put(0,14){\makebox(0,0){$2$}} 
      \put(42,14){\makebox(0,0){$2$}} 
      \put(84,14){\makebox(0,0){$2$}}
       \put(126,14){\makebox(0,0){$2$}}
       \put(170,-20){\makebox(0,0){$1$}}
       \put(170,20){\makebox(0,0){$1$}}
    \end{picture}
    \put(0,-32){\makebox(0,0){$L(t,t^{-1})\otimes (sp(m,\mathbb{R})\oplus so(p,q))\oplus \mathbb{R}iC\oplus\mathbb{R}iD$}}
\end{displaymath}

\bigskip

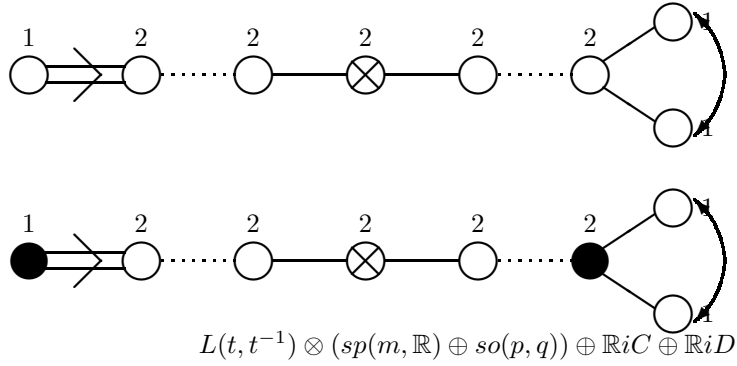
Real forms of $D^{(1)}(2,1;\alpha)$

\begin{tabular}{cc}

 \begin{picture}(80,50) \thicklines

 \put(0,0){\circle{14}} \put(42,0){\greydot} \put(84,0){\circle{14}}\put(42,-41){\circle{14}}
 
 \put(7,0){\line(1,0){28}} \put(49,0){\line(1,0){28}} \put(42,-6){\line(0,-1){28}}
 \put(0,16){\makebox(0,0){$1$}}\put(42,16){\makebox(0,0){$2$}}\put(84,16){\makebox(0,0){$1$}}\put(42,-57){\makebox(0,0){$1$}}
    \end{picture}
& \hspace{1cm}

 \begin{picture}(80,50) \thicklines

 \put(0,0){\circle{14}} \put(42,0){\greydot} \put(84,0){\circle{14}}\put(42,-41){\circle*{14}}
 
 \put(7,0){\line(1,0){28}} \put(49,0){\line(1,0){28}} \put(42,-6){\line(0,-1){28}}
 \put(0,16){\makebox(0,0){$1$}}\put(42,16){\makebox(0,0){$2$}}\put(84,16){\makebox(0,0){$1$}}\put(42,-57){\makebox(0,0){$1$}}
    \end{picture}
\hspace{1cm}
 \begin{picture}(80,50) \thicklines

 \put(0,0){\circle{14}} \put(42,0){\greydot} \put(84,0){\circle{14}}\put(42,-41){\circle{14}}
 \qbezier(42,25)(21, 24)(0, 10)
 \qbezier(42,25)(63, 24)(84, 10)
 \put(0,10){\vector( -2, -1){0}}
\put(84,10){\vector( 2, -1){0}}
 \put(7,0){\line(1,0){28}} \put(49,0){\line(1,0){28}} \put(42,-6){\line(0,-1){28}}
  \put(0,16){\makebox(0,0){$1$}}\put(42,16){\makebox(0,0){$2$}}\put(84,16){\makebox(0,0){$1$}}\put(42,-57){\makebox(0,0){$1$}}
  \put(0,-70){\mbox{$sl(2,\mathbb{C}\oplus sl(2,\mathbb{R})$}}
  \put(-120,-70){\mbox{$su(2)\oplus su(2)\oplus sl(2,\mathbb{R})$}}
  \put(-250,-70){\mbox{$sl(2,\mathbb{R})\oplus sl(2,\mathbb{R})\oplus sl(2,\mathbb{R})$}}
    \end{picture}
\end{tabular}

\vspace{3cm}

Real forms of $C^{(1)}n$
\begin{displaymath}
 \begin{picture}(80,50) \thicklines
  
  \put(0,0){\circle{14}}\put(42,0){\circle{14}}\put(84,0){\circle*{14}}
    \put(-32,21){\greydot}\put(-32,-21){\greydot}
    
    \put(-5,5){\line(-3,2){21}} \put(-5,-5){\line(-3,-2){21}}\put(7,0){\dottedline{4}(1,0)(28,0)}\put(42,0){\vertexcn}
  \put(-28,-14.5){\line(0,1){30}}\put(-36,-14.5){\line(0,1){30}}
 \put(0,-32){\makebox(0,0){$L(t,t^{-1})\otimes (sp(n,\mathbb{R})\oplus so(2))\oplus \mathbb{R}ic\oplus\mathbb{R}id$}}
 \put(0,15){\makebox(0,0){$2$}}\put(42,15){\makebox(0,0){$2$}}\put(85,15){\makebox(0,0){$1$}}
 \put(-32,32){\makebox(0,0){$1$}}\put(-42,-22){\makebox(0,0){$1$}}
 \end{picture}
\end{displaymath}

\begin{displaymath}
 \begin{picture}(80,50) \thicklines
  
  \put(0,0){\circle*{14}}\put(42,0){\circle{14}}\put(84,0){\circle{14}}
    \put(-32,21){\greydot}\put(-32,-21){\greydot}
    
    \put(-5,5){\line(-3,2){21}} \put(-5,-5){\line(-3,-2){21}}\put(7,0){\dottedline{4}(1,0)(28,0)}\put(42,0){\vertexcn}
  \put(-28,-14.5){\line(0,1){30}}\put(-36,-14.5){\line(0,1){30}}
 \put(0,-32){\makebox(0,0){$L(t,t^{-1})\otimes (sp(r,s)\oplus so(2))\oplus \mathbb{R}ic\oplus\mathbb{R}id$}}
 \end{picture}
\end{displaymath}

\vspace{2cm}
Real forms of $F^{(1)}(4)$
\begin{displaymath}
 \begin{picture}(80,40) \thicklines
   \put(-42,0){\circle{14}}\put(0,0){\greydot}\put(42,0){\circle{14}}\put(84,0){\circle*{14}}
  \put(42,0){\vertexgcn}\put(-42,0){\vertexbbn}\put(6,0){\line(1,0){29}}
  \put(20,-20){\makebox(0,0){$L(t,t^{-1})\otimes (sl(2,\mathbb{R})\oplus g_{c})\oplus \mathbb{R}ic\oplus\mathbb{R}id$}}
 \end{picture}
\end{displaymath}

\begin{displaymath}
 \begin{picture}(80,40) \thicklines
   \put(-42,0){\circle*{14}}\put(0,0){\greydot}\put(42,0){\circle{14}}\put(84,0){\circle{14}}
  \put(42,0){\vertexgcn}\put(-42,0){\vertexbbn}\put(6,0){\line(1,0){29}}
  \put(20,-20){\makebox(0,0){$L(t,t^{-1})\otimes (sl(2,\mathbb{R})\oplus g_{s})\oplus \mathbb{R}ic\oplus\mathbb{R}id$}}
 \end{picture}
\end{displaymath}
\\

Real forms of $G^{(1)}(3)$
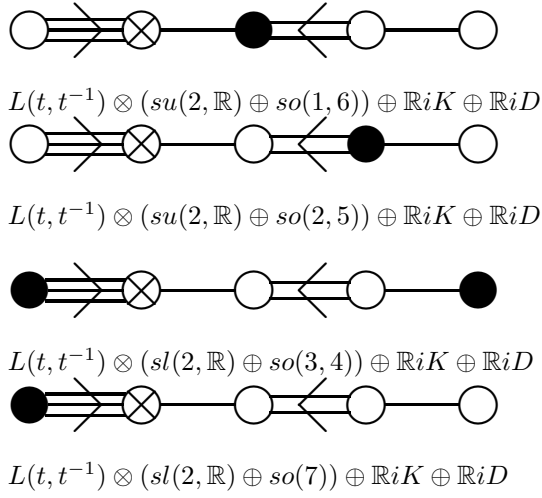
\begin{displaymath}
 \begin{picture}(80,40) \thicklines
   \put(-42,0){\circle{14}}\put(0,0){\greydot}\put(42,0){\circle*{14}}\put(84,0){\circle{14}}\put(126,0){\circle{14}}
  \put(-42,0){\vertexgbn}\put(7,0){\line(1,0){28}}\put(42,0){\vertexcn}\put(91,0){\line(1,0){28}}
     \put(-50,-30){\mbox{$L(t,t^{-1})\otimes (su(2,\mathbb{R})\oplus so(1,6))\oplus \mathbb{R}iK\oplus \mathbb{R}iD$}}
 \end{picture}
\end{displaymath}
\begin{displaymath}
 \begin{picture}(80,40) \thicklines
   \put(-42,0){\circle{14}}\put(0,0){\greydot}\put(42,0){\circle{14}}\put(84,0){\circle*{14}}\put(126,0){\circle{14}}
  \put(-42,0){\vertexgbn}\put(7,0){\line(1,0){28}}\put(42,0){\vertexcn}\put(91,0){\line(1,0){28}}
     \put(-50,-30){\mbox{$L(t,t^{-1})\otimes (su(2,\mathbb{R})\oplus so(2,5))\oplus \mathbb{R}iK\oplus \mathbb{R}iD$}}
 \end{picture}
\end{displaymath}

\begin{displaymath}
 \begin{picture}(80,40) \thicklines
   \put(-42,0){\circle*{14}}\put(0,0){\greydot}\put(42,0){\circle{14}}\put(84,0){\circle{14}}\put(126,0){\circle*{14}}
  \put(-42,0){\vertexgbn}\put(7,0){\line(1,0){28}}\put(42,0){\vertexcn}\put(91,0){\line(1,0){28}}
       \put(-50,-30){\mbox{$L(t,t^{-1})\otimes (sl(2,\mathbb{R})\oplus so(3,4))\oplus \mathbb{R}iK\oplus \mathbb{R}iD$}}
 \end{picture}
\end{displaymath}
\begin{displaymath}
 \begin{picture}(80,40) \thicklines
   \put(-42,0){\circle*{14}}\put(0,0){\greydot}\put(42,0){\circle{14}}\put(84,0){\circle{14}}\put(126,0){\circle{14}}
  \put(-42,0){\vertexgbn}\put(7,0){\line(1,0){28}}\put(42,0){\vertexcn}\put(91,0){\line(1,0){28}}
  \put(-50,-30){\mbox{$L(t,t^{-1})\otimes (sl(2,\mathbb{R})\oplus so(7))\oplus \mathbb{R}iK\oplus \mathbb{R}iD$}}
 \end{picture}
\end{displaymath}

\vspace{2cm}
\paragraph{\bf Acknowledgement:}
 Professor Meng Kiat Chuah is gratefully acknowledged for his encouragement and
reading the earlier version of the manuscript.
The author thanks National Board of Higher Mathematics, India (Project
Grant No. 48/3/2008-R\&DII/196-R) for financial support and with P.I. Prof K. C. Pati for guidance and engage me with
the problem of Vogan diagrams.


\begin{thebibliography}{13}
 
 \bibitem{batra:vogan} Batra P, {\em Vogan diagrams of affine Kac-Moody algebras}, Journal of Algebra 251, 80-97 (2002).
 \bibitem{batra:affine} Batra P, {\em Invariant of Real forms of Affine Kac-Moody Lie algebras}, Journal of Algebra 223,208-236 (2000).
\bibitem{chuah:vsuper} Chuah Meng-Kiat, {\em Cartan automorphisms and Vogan superdiagrams},Math.Z.DOI 10.1007/s00209-012-1030-z.
\bibitem{Chuah:finite}  Chuah Meng-Kiat, {\em Finite order automorphism on contragredient Lie superalgebras}, Journal of Algebra 
351 (2012) 138-159.
\bibitem{frap:hypersu} Frappat L, Sciarrino A, {\em Hyperbolic Kac-Moody superalgebras}, arXiv:math-ph/0409041v1.
\bibitem{Frap:affine} Frappat L, Sciarrino A and Sorba, {\em Structure of basic Lie superalgebras and of their affine extensions},
commun. Math. Phys. 121,457-500 (1989)
\bibitem{kac:infi}Kac V.G., {\em Infinite dimensional Lie algebras}, third edition 2003.
\bibitem{kac:sup}Kac V.G.,{\em Lie superalgebras}, Adv. Math. 26,8 (1977).
\bibitem{Parker:classification} Parker M, {\em Classification of real simple Lie superalgebras of classical type}, 
J. Math.Phys. 21(4), April 1980.
\bibitem{ransingh:vsuper} Ransingh B and Pati K C, {\em Vogan diagrams of Basic Lie superalgebra}, arXiv:1205.1394v1 [math.RT] 7 May 2012.
\bibitem{ray:char} Ray Urmie, {\em A character formula for generalized Kac-Moody superalgebras}, Journal of algebra 177, 154-163
\bibitem{Knapp:Lie groups}Knapp A.W., {\em Lie groups beyond an Introduction}, Second Edition.
\bibitem{borcherd:mtheory} Pierre Henry-Labord, ab Bernard Juliab and Louis
Paulotb,  {\em Real borcherds superalgebras and M-theory}
JHEP 04 (2003) 060.


\end{thebibliography}
\end{document}